\documentclass[11pt]{article}
\usepackage[margin=2cm]{geometry}
\usepackage{etex}
\usepackage{authblk}

\usepackage{amsmath}
\usepackage[usenames,dvipsnames,svgnames,table]{xcolor}
\usepackage{enumerate}
\usepackage{graphicx}
\usepackage{mathrsfs}
\usepackage{hyperref,comment}
\usepackage{subfig,xspace}

\usepackage{tikz}
\usepackage{pgf}
\usepackage{pgflibraryarrows}
\usepackage{pgffor}
\usepackage{pgflibrarysnakes}
\usetikzlibrary{fit} 
\usetikzlibrary{positioning}
\usepgflibrary{shapes}
\usetikzlibrary{snakes,automata}
\usetikzlibrary{shadows}

\tikzset{
  shadowed/.style={preaction={
      transform canvas={shift={(2pt,-1pt)}},draw opacity=.2,#1,preaction={
        transform canvas={shift={(3pt,-1.5pt)}},draw
        opacity=.1,#1,preaction={
          transform canvas={shift={(4pt,-2pt)}},draw
          opacity=.05,#1,
  }}}},
}

\makeatletter
\newif\if@restonecol
\makeatother
\usepackage{ntheorem}

\newtheorem{theorem}{Theorem}	
\renewtheorem*{theorem*}{Theorem}	

\newtheorem{definition}{Definition}

\newtheorem{lemma}{Lemma}

\newtheorem{remark}{Remark}
\newtheorem{example}{Example}

\usepackage{amsmath,mathrsfs}
\usepackage{amssymb}

\def\R{\mathbb{R}}

\newcommand{\U}{\mathbf{U}}

\usetikzlibrary{patterns}

\tikzset{
  ashadow/.style={opacity=.25, shadow xshift=0.07, shadow yshift=-0.07},
}

\hypersetup{colorlinks=true,linkcolor={blue},citecolor={Maroon}}
           
\newenvironment{proof}{\paragraph{Proof:}}{\hfill$\square$}

\usepackage{bbm}
\usepackage{enumerate}
\usepackage{times}
\usepackage{centernot}
\usepackage{url}
\usepackage{cite}
\usepackage{amsfonts,mathrsfs}
\usepackage{amssymb,amsmath}
\usepackage{verbatim}
\usepackage{acronym}
\usepackage{mathtools}

\usepackage{graphicx}

\usepackage{hyperref}
\usepackage{enumitem,kantlipsum}

\newcommand{\remove}[1]{}

\newcommand{\blue}[1]{\textcolor{black}{#1}}

\def\fskip#1{}
\def\P{\mathcal P}

\def\R{\mathbb{R}}
\def\E{\mathbb E}

\def\expect{\widehat}

\def\ee{\mathbf e}
\def\B{\mathcal B}

\def\x{\mathbf x}
\def\boldpsi{\boldsymbol\psi}
\def\boldpi{\boldsymbol\pi}
\def\y{\mathbf y}
\def\z{\mathbf z}
\def\v{\mathbf v}

\def\U{\mathcal U}

\def\allzero{\mathbf{0}}
\def\allone{\mathbf{1}}
\def\Pr{\text{Pr}}
\usepackage{amsmath}

\DeclarePairedDelimiterX{\infdivx}[2]{(}{)}{%
  #1\;\delimsize\|\;#2%
}

\usepackage[normalem]{ulem}

\def\Pr{\text{Pr}}
\def\N{\mathbb{N}}

\def\u{\mathbf u}

\usepackage{algorithm}
\usepackage[noend]{algpseudocode}

\begin{document}

\title{A Random Adaptation Perspective on Distributed Averaging}

 \author{Rohit Parasnis}
 \author{Ashwin Verma}
 \author{Massimo Franceschetti}
 \author{Behrouz Touri\thanks{Email: rparasni,a1verma,mfranceschetti,btouri@ucsd.edu}} 
 \affil{Department of Electrical and Computer Engineering, University of California San Diego}
\date{}

\maketitle

\begin{abstract}
We propose a random adaptation variant  of time-varying distributed averaging dynamics in discrete time. We show that this leads to novel interpretations of fundamental concepts in distributed averaging, opinion dynamics, and distributed learning. Namely, we show that the ergodicity of a stochastic chain is equivalent to the \textit{almost sure} (\textit{a.s}.) finite-time agreement attainment in the proposed random adaptation dynamics. Using this result, we provide a new interpretation for the \textit{absolute probability sequence} of an ergodic chain. We then modify the base-case dynamics into a time-reversed inhomogeneous Markov chain, and we show that in this case ergodicity is equivalent to the uniqueness of the limiting distributions of the Markov chain. Finally, we introduce and study a time-varying random adaptation version of the Friedkin-Johnsen model and a rank-one perturbation of the base-case dynamics.
\end{abstract}

\section{INTRODUCTION}
Distributed averaging is a central mechanism to information mixing in distributed optimization~\cite{nedic2007rate,rabbat2004distributed}, distributed parameter estimation and signal processing~\cite{predd2006distributed, cattivelli2009diffusion, jadbabaie2012non, lalitha2014social, nedic2016tutorial, parasnis2022non, rangi2018multi},  decentralized control of robotic networks~\cite{bullo2008distributed}, and opinion dynamics~\cite{degroot1974reaching,acemoglu2011opinion, proskurnikov2017tutorial, proskurnikov2018tutorial, hegselmann2002opinion, parasnis2021convergence}. Hence, a variety of distributed averaging dynamics have been studied till date within different mathematical frameworks~\cite{xiao2004fast, cao2011necessary,el2016design}.

In particular, distributed averaging algorithms over time-varying networks are often analyzed using chains/sequences of stochastic matrices (a class of non-negative matrices). Several properties of such chains, such as \textit{ergodicity} and \textit{reciprocity}, have been studied in detail~\cite{blackwell1945finite, chatterjee1977towards, touri2012backward, touri2012product, martin2015continuous,bolouki2015consensus,parasnis2022towards, de2022ergodicity}. Despite this abundance of literature, interpretations of some key concepts in this area, such as Kolmogorov's absolute probability sequences, have remained elusive.

On the other hand, stochastic matrices   can be interpreted by noting that every row of such   matrices gives the probability mass function of a discrete random variable. Thus motivated, we introduce a \textit{random adaptation} framework in which each entry of a stochastic matrix denotes the \textit{probability} that a node $i$ adapts to the state of a node $j$, rather than the \textit{weight} assigned by $i$ to $j$. 
\blue{That is,} we propose random adaptation variants of some classical discrete-time distributed averaging dynamics. We then study the proposed dynamics to interpret in a new light certain concepts that are central to the study of stochastic chains and averaging dynamics. Our contributions are as follows:
\begin{enumerate}[leftmargin=0cm,itemindent=.5cm,labelwidth=\itemindent,labelsep=0cm,align=left]
    \item  \textbf{\textit{An Interpretation of Ergodicity:}} 
   
    We show that the classical notion of ergodicity is \textit{equivalent} to an intuitive condition of agreement in finite time in the random adaptation dynamics. \blue{This differs from the classical scenario of distributed averaging dynamics, where ergodicity does not necessarily ensure finite-time agreement.} 
    \item \textit{\textbf{An Interpretation of Absolute Probability Sequence:}} Using the above characterization of ergodicity, we interpret the \textit{absolute probability sequence} of an ergodic chain as the  limiting probability distribution of the common value attained by all the nodes in the proposed dynamics.
    \item \textbf{\textit{Ergodicity vis-a-vis Uniqueness of the Limiting Distribution:}} We propose a time-reversed, transposed variant of the aforementioned dynamics and use it to show that the limiting distribution of the state vector is unique if and only if the given  stochastic chain is ergodic. \blue{This leads to a new insight: absolute probability sequences are to ergodic chains what stationary distributions are to regular stochastic matrices.}
    \item \textit{\textbf{Asymptotic Behaviors of Variants:}} Finally, we discuss the random adaptation interpretation of the well-known Friedkin-Johnsen model of opinion dynamics~\cite{friedkin1990social} as well as rank-one perturbations of  the base-case dynamics. In both the cases, we study the limiting probability distributions of the agents' states/opinions.
\end{enumerate}
In addition, in Section~\ref{sec:simulations} we provide numerical examples to illustrate and interpret some of our results on the base-case dynamics as well as on our random adaptation variants of the Friedkin-Johnsen model.

\blue{\textit{Related works:} This paper is closely related to the voter model~\cite{clifford1973model,holley1975ergodic} and its many extensions (e.g.,~\cite{sood2005voter,schneider2009generalized,castellano2009nonlinear,granovsky1995noisy,congleton2004median,mobilia2007role}), which apply to social networks and describe the processes of individuals randomly adapting to their neighbors' political preferences. However, our work differs from these prior works in at least two fundamental ways. First, the voter model and most of its variants assume a limit on the number of agents that update their opinions synchronously, whereas our model does not make any assumptions on the level of synchrony among the agents' updates (note that our update rules can be used to prevent an arbitrary set of  agents from updating their opinions by simply setting the corresponding diagonal entries of the associated stochastic matrices to 1). For instance, the voter model with zealots~\cite{acemoglu2013opinion} differs from our random adaptation variant of the Friedkin-Johnsen model in that, in the former, at most one agent updates its state in any given update period. Second, unlike this paper, none of the prior works aims to unravel the properties of the stochastic chains or matrices associated with the random dynamics proposed therein.}

Besides, some of our results may be {related to} 
the well-studied {duality} between coalescing random walks and voter dynamics~\cite{cox1989coalescing,aldous1995reversible}.

\textit{Notation: } In this paper $\N:=\{1,2,\ldots\}$, $\N_0:=\N\cup\{0\}$,  $\R$ denotes the set of real numbers, $\R^n$ denotes the set of $n$-dimensional real-valued column vectors, and $\R^{n\times n}$ denotes the set of $n\times n$ real-valued matrices. For $n\in\N$, we let ${[n]:=\{1,2,\ldots,n\}}$. For a vector $\blue{\v}\in\R^n$, $v_i$ denotes its $i$th entry, and $b_{ij}$ denotes the $(i,j)$th entry of a matrix ${B\in\R^{n\times n}}$. All matrix and vector inequalities are assumed to hold entry-wise.

For a given $n\in\N$ \blue{whose value is clear from the context}, let $O,I$ denote the $n\times n$ matrix with all zero entries and the identity matrix, respectively. We denote the column vectors with all zero entries and all one entries in $\R^n$ by $\allzero$ and $\allone$, respectively. $\blue \ee^{(i)}\in\R^n$ denotes the $i$th canonical basis vector. 

For a matrix $A\in\R^{n\times n}$ and a set $S\subseteq [n]$, let $A_S$ be the principal sub-matrix of $A$ corresponding to the rows and columns indexed by $S$. Let $\bar S:=[n]\setminus S$ be the complement of $S$ with respect to $[n]$, and for sets $S,T\subseteq [n]$, let $A_{ST}$ denote the sub-matrix of $A$ corresponding to the rows indexed by $S$ and the columns indexed by $T$.

We say $\blue{\v}\in\R^n$ is \textit{stochastic} if $\v\geq \allzero$ and $\allone^T\v=1$, where the superscript $^T$ denotes transposition. A matrix $A\in\R^{n\times n}$ is called \textit{row-stochastic} or simply \textit{stochastic} if each row of \blue{$A$} is stochastic. We let $\P_n$ (respectively, $\P_{n\times n}$) denote the set of all stochastic vectors (respectively, matrices) in $\R^n$ (respectively, $\R^{n\times n}$). For a sequence $\{A(t)\}_{t=0}^\infty$ and indices $t_1<t_2$, we let
$
    {A(t_2:t_1):=A(t_2-1)A(t_2-2)\cdots A(t_1)}
$
with the convention that $A(t:t):=I$ for all $t\in\N_0$. 

All random objects in this work are defined with respect to an underlying probability space $(\Omega,\B,\Pr)$. With an abuse of terminology, by $\blue{Z}(t)\in \R^n$ for a random process $\{\blue{Z}(t)\}$ we mean $\blue{Z}(t,\omega)\in \R^n$ for all $t\in \N_0$ and $\omega\in\Omega$. \blue{We use }$\E[\cdot]$ \blue{to }denote the expectation operator with respect to this probability space. We use the notation $\expect Z$ to denote the  expectation of a random variable/matrix/vector $Z$.

\section{RANDOM ADAPTATION DYNAMICS}~\label{sec:formulation}
In \blue{distributed} averaging dynamics, we are given a sequence of stochastic matrices $\{{Q}(t)\}\blue{_{t=0}^\infty}$, and we are interested in studying the dynamics
\begin{align}\label{eqn:averaging}
    \blue{\x}(t+1)=Q(t)\blue\x(t), 
\end{align}
for some initial time $t_0\in \N_0$ and initial \blue{state} $\x(t_0)\in \R^n$. As mentioned in the introduction, each row of $Q(t)$ is a stochastic vector that can be viewed as the probability mass function of a certain random variable taking $n$ values. Our random adaptation viewpoint formulates a very natural sequence of random variables that exhibits this behavior. Consider $n$ agents assuming a random state $\blue{x}_i(t)\in\R$ (more precisely, $\blue{x}_i(t,\omega)\in\R$ where $\omega\in \Omega$) evolving over discrete time $t\in\N_0$. Let the starting time $t_0\in\N_0$ and the initial states $\blue{\x}(t_0)\in \R^n$ be an arbitrary deterministic vector. For a given sequence $\{Q(t)\}$ of stochastic matrices,  the random adaptation scheme is \blue{defined} as follows: At time $t\geq t_0$, agent $i$ adopts agent $j$'s state with probability $q_{ij}(t)$ independently of other agents as well as every agent's past choices and states, i.e., ${\Pr(\blue{x}_i(t+1) = \blue{x}_j(t)\mid \blue{\x}(t),\ldots, \blue{\x}(t_0))=q_{ij}(t)}$ for all $i,j\in[n]$. Therefore, we can write 
\begin{align}\label{eq:main}
    \blue{\x}(t+1) = A(t) \blue{\x}(t),
\end{align}
where $\blue{\x}(t)\in\R^n$ is the random state vector (whose $i$-th entry is $\blue{x}_i(t)$) and $A(t)\in\R^{n\times n}$ is a binary random stochastic matrix, i.e., $\{a_{ij}(t):i,j\in [n]\}$ are Bernoulli random variables with parameters $\expect a_{ij}(t)=q_{ij}(t)$ for all $i,j\in[n]$.

We now observe a few important properties of~\eqref{eq:main}: (i) The random process $\{\blue\x(t)\}_{t=0}^\infty$ is a non-homogeneous Markov chain with state space size at most $n^n$ as the state of each agent at every time \blue{instant} is in $\{x_i(t_0)\mid i\in [n]\}$. (ii) For each $t\in\N_0$, the rows of $A(t)$ are independent random vectors. (iii) The random matrices $\{A(t)\}_{t=0}^\infty$ are independent and hence, for each $t\in\N$, $A(t)$ is independent of $\blue\x(t)$.
\section{MAIN RESULTS}\label{sec:main_result}
We first define two properties that will be shown to be closely related. The first property relates to stochastic chains.


\begin{definition} [\textbf{Ergodicity\blue{~\cite{chatterjee1977towards,touri2012product}}}] \label{def:ergodicity} A deterministic (non-random) stochastic chain $\{Q(t)\}_{t=0}^\infty$ is said to be \textit{ergodic} if, for every $t_0\in\N_0$, there exists a vector $\blue \boldpsi (t_0)\in\P_n$ 
such that $\lim_{t\to\infty} Q(t:t_0)=\allone  \blue\boldpsi^T(t_0)$.
\end{definition}

Put differently, a stochastic chain $\{Q(t)\}_{t=0}^\infty$ is called ergodic if every backward matrix product $Q(t:t_0)$ converges to a rank-one matrix that has identical rows.

The second property relates to the random adaptation dynamics~\eqref{eq:main}. 
\begin{definition} [\blue{\textbf{Finite Agreement}}]
We say that the adaptation dynamics~\eqref{eq:main} has an \textit{a.s.} finite agreement property if for all initial time\blue{s} $t_0\in \N_0$ and initial \blue{states} $\blue\x(t_0)\in\R^n$, \blue{there exist a random scalar $y= y(t_0, \x(t_0))$ and a random (stopping) time $T\geq t_0$ such that $\x(t)=y\allone$ for all $t\geq T$}. 
\end{definition}

In other words, the random adaptation dynamics~\eqref{eq:main} has the finite agreement property if at \blue{some} time after the initiation of the dynamics, all the agents \blue{adopt} the same state. 

{Note that the ergodicity of $\{Q(t)\}$ does not imply finite agreement for the deterministic dynamics \eqref{eqn:averaging}. For example, suppose $n=2$, let $\x(0)={\ee}^{(1)}$, and consider the static chain defined by
$
    Q(t) =
    \begin{pmatrix}
        p & 1-p\\
        1-q & q\\
    \end{pmatrix}
$
for all $t\in\N_0$, where $p,q\in (0,1)$. Then it can be verified that $\{Q(t)\}_{t=0}^\infty$ is ergodic. However, we also have ${ x_1(t)-x_2(t)=(p+q-1)^t}$ for all $t\in\N$, which implies that no agreement is reached in finite time unless $p+q=1$.

On the contrary, if $\blue\x(t_0)\in \{ \blue\ee^{(1)},\blue\ee^{(2)}\}$ the random adaptation variation \eqref{eq:main} of the same static chain ensures that the two agents reach an agreement at time ${T:=\inf\{t\geq t_0: x_1(t)\oplus x_2(t)=0\}}$, where $\oplus$ denotes \blue{the} \textit{exclusive OR} operation. Since $T$ is a geometric random variable with the parameter ${p(1-p)+q(1-q)>0}$, it is finite \textit{a.s.} In fact, even for $n\neq 2$, every ergodic chain guarantees finite agreement for all initial conditions, as we now show.}

\begin{theorem}\label{thm:main}
The random adaptation dynamics~\eqref{eq:main} has the finite agreement property if and only if ${\{\expect A(t)\}_{t=0}^\infty=\{Q(t)\}_{t=0}^\infty}$ is an ergodic chain.
\end{theorem}

\begin{proof}
We establish ergodicity first as a sufficient condition and then as a necessary condition for \textit{a.s.\ }finite agreement with an arbitrary initial condition $(t_0,\blue\x(t_0))$.

\subsubsection*{Proof of Sufficiency}
It suffices to show that an agreement occurs in the network infinitely often \textit{a.s}., because once an agreement is reached, the state vector stops evolving in time.

We first note that by Definition~\ref{def:ergodicity}, there exists a ${ \blue\boldpsi(t_0)\in\P_n}$ such that ${\lim_{t\to\infty} Q(t:t_0)=\allone  \blue\boldpsi^T(t_0)}$ \textit{a.s}. As ${ \blue\boldpsi^T(t_0)\allone=1}$, there exists an index $\ell(t_0)\in[n]$ such that ${\blue\psi_{\ell(t_0)}(t_0)\geq \frac{1}{n}}$, which implies that ${\lim_{t\to\infty} Q_{i\ell (t_0)}(t:t_0)\geq \frac{1}{n}}$ for all $i\in [n]$. Consequently, there exists a time $t_1\geq t_0$ such that $ Q_{i\ell(t_0)}(t_1:t_0)\geq \frac{1}{2n}$ for all $i\in [n]$. By Lemma~\ref{lem:positive_correlation}, 
 this implies that $\Pr\left( A_{[n]\{\ell(t_0)\}}(t_1:t_0) =\allone\right)\geq \left(\frac{1}{2n}\right)^{n}$. Since $A(t_1:t_0)\in\P_{n\times n}$ is binary, it follows that ${\Pr(A(t_1:t_0)=\allone (\blue\ee^{\ell(t_0)})^T)\geq \left(\frac{1}{2n}\right)^n}$. 

As $t_0$ is arbitrary, we can repeat the above analysis with different starting times and obtain an increasing sequence of times $\{t_k\}_{k=0}^\infty$ such that ${\Pr(A(t_{k+1}:t_{k})=\allone (\blue\ee^{\ell(t_k )})^T)\geq \left(\frac{1}{2n}\right)^n}$.  Since $\blue\x(t_{k+1})=A(t_{k+1}:t_k)\blue\x(t_k)$, this further implies that $\Pr(\blue\x(t_{k+1})=\blue x_{\ell}(t_k)\allone)\geq\left(\frac{1}{2n}\right)^n$. As a result, letting $C_k$ denote the event that an agreement exists in the network at time $t_k\geq t_0$, we have $\sum_{k=0}^\infty\Pr(C_k)=\infty$. Now,  $\{C_k\}_{k=0}^\infty$ are independent events because $\{A(t)\}_{t=0}^\infty$ are independent and
$\{[t_k,t_{k+1}-1]:k\in\N_0\}$ are disjoint intervals. Therefore, by the Second Borel-Cantelli Lemma  {\cite[Theorem 2.3.6]{durrett2019probability}}, 
infinitely many events among $\{C_k\}_{k=0}^\infty$ occur \textit{a.s.}, which proves the assertion.
\subsubsection*{Proof of Necessity}
Suppose there exist ${T=T(t_0,\blue\x(t_0))<\infty}$ and $\blue y=\blue y(t_0,x(t_0))\in\R$ such that $\blue\x(t)=\blue y\allone$ \textit{a.s.} for all $t\geq T$. Then $\blue\x(t)=A(t:t_0)\blue\x(t_0)$ implies that ${\lim_{t\to \infty} A(t:t_0)\blue\x(t_0) = \blue y\allone}$ \textit{a.s.} 
 
Besides, we know that $\|\blue\x(t)\|_{\infty}\leq \|\blue\x(t_0)\|_{\infty}$ for all ${t \geq t_0}$. Therefore, by  the Dominated Convergence Theorem~\cite[Theorem 1.6.7]{durrett2019probability}, 
we have
\begin{align}\label{eq:new}
    &\E\left[\blue y\right]\allone=\E\left[ \lim_{t\to \infty} A(t:t_0)\blue\x(t_0)\right]\cr
      &=\lim_{t\to \infty} \E[A(t:t_0)\blue\x(t_0)]=\lim_{t\to\infty} \expect A(t:t_0) \blue\x(t_0),
\end{align}
where the last equality follows from the independence of $\{A(t)\}_{t=0}^\infty$. As a result, 
    ${\lim_{t\to\infty}Q(t:t_0)  \blue\x(t_0) = \expect y \allone}$. 
For the initial condition $\blue\x(t_0)=\blue\ee^{(i)}$, this implies that the $i$th column of $Q(t:t_0)$ converges to  $\psi_i\allone$ for some scalar $\psi_i\in \R$ and hence, $\lim_{t\to\infty} Q(t:t_0)=\allone \blue\boldpsi^T$ for some vector $\blue\boldpsi=\blue\boldpsi(t_0)\in \P_n$. The latter step follows from the fact that the set of row-stochastic matrices is a closed semigroup (under matrix multiplication).  
\end{proof}
\begin{remark} \label{rem:matrix_limit}
  Theorem~\ref{thm:main} enables us to comment further on ergodic chains. To elaborate, we can repeat some of the arguments used in the proof above to show that if $\{Q(t)\}_{t=0}^\infty$ is ergodic, then for all ${(t_0,\blue \x(t_0))\in\N_0\times\R^n}$, there \textit{a.s.\ }exists a $\blue\boldpi(t_0)\in\P^n$ such that $\lim_{t\to\infty} A(t:t_0)=\allone \blue\boldpi^T(t_0)$ for all $t_0\in\N_0$. Moreover,  $\{A(t)\}_{t=0}^\infty\in\P_{n\times n}$ being binary implies that $\blue\boldpi(t_0)$ is binary, i.e.,  ${\blue\boldpi(t_0)\in\{\blue\ee^{(i)}:i\in[n]\}}$. Finally, taking expectations on both sides yields ${\lim_{t\to\infty}\expect A(t:t_0)=\allone\expect{\blue\boldpi}^T(t_0)}$, i.e., ${\lim_{t\to\infty}Q(t:t_0)=\allone\expect{\blue\boldpi}^T(t_0)}$ where ${\expect{\blue\boldpi}^T(t_0)\in\P_n}$. Interestingly, for the chain $\{Q (t)\}_{t=0}^\infty$, one can verify  that $\{\expect{\blue\boldpi}(t)\}_{t=0}^\infty$ forms what we call an \textit{absolute probability sequence}, a concept defined below and introduced by Kolmogorov in~\cite{kolmogoroff1936theorie}.
\end{remark}

\begin{definition}[Absolute Probability Sequence]
For a deterministic stochastic chain $\{Q(t)\}_{t=0}^{\infty}$, a sequence of stochastic vectors $\{\blue\boldpsi(t)\}_{t=0}^{\infty}$ is said to be an absolute probability sequence if $\blue\boldpsi^{T}(t+1)Q(t) = \blue\boldpsi^T(t)$  for all $t \geq 0$.
\end{definition}

We now connect this novel concept with the dynamics~\eqref{eq:main}.

\begin{theorem}
Suppose that $\{Q(t)\}_{t=0}^\infty= \{\expect A(t)\}_{t=0}^{\infty}$ is ergodic for the dynamics \eqref{eq:main}, with an absolute probability sequence $\{\blue\boldpsi(t)\}_{t=0}^{\infty} = \{\expect{ \blue\boldpi}(t)\}_{t=0}^\infty$, where $\{\blue\boldpi(t)\}_{t=0}^\infty$ is an absolute probability sequence for $\{A(t)\}_{t=0}^\infty$. Let $\blue y=\blue y(t_0,\blue\x(t_0))$ be the agreed value of \blue{all }the agents, i.e.,  $\lim_{t\to \infty} \blue\x(t)=\blue y\allone$ \textit{a.s} for initial conditions ${(t_0,\blue\x(t_0))\in\N_0\times \R^n}$ such that $\{\blue x_i(t_0)\}_{i=1}^n$ are all distinct. Then the probability distribution of $\blue y$ is given by 
    ${p_i(t_0):=\Pr( \blue y = \blue x_i(t_0)) = \psi_i(t_0)}$ for all ${i \in [n]}$.
\end{theorem}
\begin{proof}
By Remark~\ref{rem:matrix_limit}, we almost surely have ${\lim_{t\to\infty} \blue\x(t)=\lim_{t\to\infty} A(t:t_0) \blue\x(t_0)=\allone \blue\boldpi^T(t_0)\blue\x(t_0)}$. Thus, ${ \blue y=\blue\boldpi^T(t_0)\blue\x(t_0)}$, which implies that ${\expect {\blue y}=\blue{\expect \boldpi}^T(t_0)\blue \x(t_0)}=\sum_{i=1}^n  \psi_i(t_0)x_i(t_0)$.

On the other hand, the definition of expectation implies that $\blue{\expect y}=\sum_{i=1}^n p_i(t_0) x_i(t_0)$.

Hence, ${\sum_{i=1}^n \psi_i(t_0)\blue x_i(t_0) = \sum_{i=1}^n p_i(t_0) \blue x_i(t_0)}$. Since this holds for all $\blue\x(t_0)\in\R^n$, we must  have $p_i(t_0) = \psi_i(t_0)$ for all $i\in[n]$.
\end{proof}

\section{VARIANTS AND EXTENSIONS}

\subsection{Time-Reversed  Non-homogeneous Markov Chains}
Let $\{A(t)\}_{t=0}^\infty$ be a random sequence of independent binary matrices, and let ${\{\blue\ee^{(i)}:i\in [n]\}}$ be the state space of a {time-reversed} Markov chain whose probability transition matrix is $Q(t):=\expect A(t)$ at time $t$. To be precise, the Markov chain is a random process $\{\blue\z(t)\}$ that starts at an arbitrary time instant $t_\infty\in\N$ with an arbitrary probability distribution given by $\blue{\mathbf p}_\infty\in\P_n$ (where ${(\blue{\mathbf p}_\infty)_i:=\Pr(\blue\z(t_\infty)= \blue\ee^{(i)})}$ and evolves backwards in time with ${\Pr(\blue\z(t)=\blue\ee^{(j)}\mid \blue\z(t+1)=\blue\ee^{(i)}) =q_{ij}(t)}$ for all ${t<t_\infty}$. Equivalently, 
\begin{align}\label{eq:time_rev}
    \blue\z^T(t)=\blue\z^T(t+1)A(t)
\end{align}
for all $t\in\{t_\infty-1,t_\infty-2, \ldots, 0\}$. Note that~\eqref{eq:time_rev} is nothing but a time-reversed, transposed variant of~\eqref{eq:main}.

To relate these dynamics to time-homogeneous chains, recall that the limiting probability distribution of a regular Markov chain is a stationary distribution independent of the initial distribution~\cite{lalley}.
Analogously, we ask, is the limiting distribution of a time-reversed inhomogeneous Markov chain an absolute probability sequence of the associated stochastic chain that is independent of $\blue{\mathbf p}_\infty$? As we show, the answer is yes if and only if the stochastic chain is ergodic.

\begin{theorem} \label{thm:time-rev}
Consider the dynamics~\eqref{eq:time_rev} with a variable starting time $t_\infty$. Let $\{\blue \boldpsi(t)\}_{t=0}^\infty$ be an absolute probability sequence for $\{Q(t)\}_{t=0}^\infty$. Then the limiting distribution $\blue{\mathbf p}(t):=\sum_{i=1}^n p_i(t) \blue\ee^{(i)} \in\P_n$ with $p_i(t):=\lim_{t_\infty\to\infty}\Pr(\blue\z(t)=\blue\ee^{(i)})$ exists and  is invariant w.r.t.\ the initial distribution $\blue{\mathbf p}_\infty$ for all $t\in\N_0$  if and only if $\{Q(t)\}_{t=0}^\infty$ is ergodic, in which case $\blue{\mathbf p}(t)=\blue\boldpsi(t)$ for all $t\in\N_0$.
\end{theorem}

\begin{proof}
    We first note that for all $t\leq t_\infty$, we have $\Pr(\blue\z(t)=\blue\ee^{(i)})=\E[\blue z_i(t)]=\blue{\expect z}_i(t)$ for all ${i\in [n]}$, which means that $ p_\infty = \blue{\expect \z}(t_\infty)$ and more generally that the probability distribution of $\blue \z(t)$ is determined by $\blue{\expect \z}^T(t) = \blue{\expect \z}^T(t_\infty)\expect A(t_\infty:t) = \blue{\mathbf p}_\infty^T Q(t_\infty:t)$ for all $t\leq t_\infty$. Thus, $\blue{\mathbf p}^T(t)=\blue{\mathbf p}_\infty^T\lim_{t_\infty\to\infty}Q(t_\infty:t)$ (if the limit exists). Using this, we first establish the sufficiency  of ergodicity and then its necessity for the invariance assertion to hold.
    
    If $\{Q(t) \}_{t=0}^\infty$ is ergodic, then $\blue{\mathbf p}^T(t)$ is given by $\blue{\mathbf p}_\infty^T\lim_{t_\infty\to\infty}Q(t_\infty:t)\stackrel{(a)}=\blue{\mathbf p}_\infty^T\allone \blue\boldpsi^T(t)\stackrel{(b)}=\blue\boldpsi^T(t)$, where  $(a)$ follows from Remark~\ref{rem:matrix_limit} and $(b)$ holds because $\blue{\mathbf p}_\infty\in\P_n$. Since $\{ \blue\boldpsi(t)\}_{t=0}^\infty$ are unique (see~\cite{blackwell1945finite}, Theorem 1), it follows that $\blue{\mathbf p}(t)=   \blue\boldpsi(t)$ \textit{a.s.\ }does not vary with $\blue{\mathbf p}_\infty$. 
    
    On the other hand, if $\{Q(t) \}_{t=0}^\infty$ is not ergodic, then there exists a $t_0\in \N_0$ such that either ${\lim_{t_\infty\to\infty} Q(t_\infty:t_0)}$ does not exist (in which case there is nothing to prove), or there exists an index $\ell\in [n]$ such that the column vector ${\blue{\v}:=\lim_{t_\infty\to\infty} Q_{[n]\,\{\ell\}}(t_\infty:t_0)}$ satisfies $\blue{\v}\neq \alpha \allone$ for all ${\alpha\in \R}$. Therefore, we can write $\blue{\v}=\alpha \allone+\beta \blue{\mathbf w}$ for some $\alpha,\beta>0$ and \blue{some} $\blue{\mathbf w}$ with $\blue{\mathbf w}^T \allone=0$. Note that ${\blue{\mathbf w}^T\blue{\v}\not=0}$.
    Note \blue{also} that for small enough  $\tilde{\beta}>0$, $\tilde{\blue{\mathbf w} }={\frac{1}{n}\allone+\tilde{\beta}\blue{\mathbf w}\in\P_n}$. Now, for  $\blue{\mathbf p}_\infty=\frac{1}{n}\allone$ and $\blue{\mathbf p}_\infty=\tilde{\blue{\mathbf w}}$, \blue{the  value of }  ${ p}_\ell(t_0)=\left(\blue{\mathbf p}_\infty^T\lim_{t_\infty\to\infty} A(t_\infty:t_0)\right)_\ell$ would be $\frac{1}{n}\allone^T\blue{\v}$ and $\frac{1}{n}\allone^T\blue{\v}+\tilde{\beta}\blue{\mathbf w}^T\blue{\v}$, respectively, which along with $\blue{\mathbf w}^T\blue{\v}\neq 0$ violate\blue{s} the invariance condition.  
    \end{proof}
\blue{Theorem~\ref{thm:time-rev} also shows that, just as stationary distributions are the limiting probability distributions of Markov chains defined by regular matrices, absolute probability sequences can be interpreted as the limiting distributions of time-reversed Markov chains defined by ergodic stochastic chains.}
\subsection{Random Adaptation Approach to Friedkin-Johnsen Model}\label{subsec:rv_fj}
The  dynamics~\eqref{eqn:averaging} can be viewed as a time-varying version of the French-Degroot opinion dynamics model where agent opinions move towards convex combinations of other agents' opinions. The Friedkin-Johnsen model, in addition to being partly influenced by neighbors, introduces a \textit{prejudice} that affects the agents' opinions. Mathematically, 
\begin{align}\label{eq:fj}
    \blue\x(t+1) = \Lambda W(t)\blue\x(t) + (I-\Lambda) \blue\u,
\end{align}
where $\blue\x(t)\in\R^n$ denotes the vector of opinions, $\blue\u\in\R^n$ is the  vector of the agents' prejudices,  $W(t)\in\R^{n\times n}$ denotes the \textit{influence matrix}, which describes how the agents influence each other, and $\Lambda\in\R^{n\times n}$ is a diagonal matrix whose $i$th diagonal entry, $\lambda_i\in[0,1]$, denotes the \textit{susceptibility} of agent $i$ to social influence and $1-\lambda_i$ denotes the susceptibility of agent $i$ to her prejudice $\blue u_i$.

Similar to~\eqref{eq:fj}, we can provide a random adaptation variation of the Friedkin-Johnsen model as follows. In the $t$-th time  period, agent $i$ decides between adapting to \blue{her} neighbor's opinion versus adapting to \blue{her} prejudice. Her choice is independent of other agents' choices and her own past choices. With a probability $\gamma(t) \in(0,1)$, she follows the adaptation scheme described earlier, and with probability $1- \gamma(t)$, she resets her opinion to her prejudice $u_i\in\R$. As before, we assume the initial state vector $\blue\x(0)\in\R^n$ to be arbitrary.  This results in the update rule
\begin{align}\label{eq:random_fj}
    \blue\x(t+1) = \Lambda(t) A(t) \blue\x(t) + (I-\Lambda(t))\blue\u, 
\end{align}
where $\blue\x(t)$ and $A(t)$ have their usual meanings, $\blue\u\in\R^n$ is a vector of external influences/prejudices, and $\{\Lambda(t)\}_{t=0}^\infty$ is a sequence of diagonal matrices whose diagonal entries $\{\lambda_i(t):i\in[n]\}$ are Bernoulli random variables with $\Pr(\lambda_i(t)=1)=\gamma_i(t)$.

\begin{remark}~\label{rem:proof_aid}Observe that~\eqref{eq:random_fj} is a special case of~\eqref{eq:main} by letting $\blue\y^T(t):=[\blue\x^T(t)\,\, \blue\u^T]$ and \blue{by} noting \blue{that}
\begin{align}\label{eq:higher_dim}
    \blue\y(t+1) = B(t) \blue\y(t),
\end{align}
where $B(t):= 
    \begin{pmatrix}
        \Lambda(t) A(t)\, & I -  \Lambda (t)\\
        O_{n\times n} \, & I_{n\times n}
    \end{pmatrix}$.
As a result, for any $t,t_0\in\N_0$ with $t\geq t_0$, we have $\blue\y(t)=B(t:t_0)\blue\y(t_0)$ with $B_{[n]}(t:t_0)= P(t:t_0)$, where $P(t):=\Lambda(t)A(t)$. Note also that $P(t)\leq A(t)$ because $\Lambda(t)\leq I$.
\end{remark}

We now define two terms: \textit{dominance in expectation} and \textit{simultaneously malleable} agents, and we  show that, if  $\{Q(t)\}_{t=0}^\infty$ is an ergodic chain with simultaneously malleable agents that dominate in expectation, then all the agents' opinions will almost surely enter  the \blue{prejudice set} (the set of external influences) $\U:=\{u_i:i\in[n]\}$ in finite time.

\begin{definition} [\textbf{Dominance in Expectation}] The agents of a set $S\subseteq[n]$ are said to \textit{dominate in expectation} if $\sum_{t=0}^\infty \allone^T Q_{\bar S S}(t)\allone=\infty$ and $\sum_{t=0}^\infty \allone^T  Q_{S\bar S}(t)\allone<\infty$.
\end{definition}

In the the average-case scenario, if a set of agents $S\subseteq [n]$ dominate in expectation, then the agents of $S$ significantly influence the rest of the agents $\bar S$ without themselves being significantly influenced by $\bar S$ in the long run.

\begin{definition} [\textbf{Simultaneously Malleable Agents}]\label{def:sma} The agents of a set $S\subseteq[n]$ are said to be \textit{simultaneously malleable} if $\sum_{t=0}^\infty \prod_{i\in S}(1-\gamma_i(t))=\infty$.
\end{definition}

Essentially, simultaneously malleable agents are those whose probability of simultaneously adapting to their respective external influences does not vanish too fast with time.
\begin{theorem} \label{thm:malleable} For the dynamics~\eqref{eq:random_fj}, suppose ${\{Q(t)\}_{t=0}^\infty=\{\expect A(t)\}_{t=0}^\infty}$ is ergodic, and suppose there exists a set $S\subseteq[n]$ of simultaneously malleable agents that dominate in expectation. Then there \textit{a.s.\ }exists a time $T<\infty$ such that $\{\blue x_i(T)\}_{i=1}^n\subseteq\U$.
\end{theorem}

\begin{proof}
First, note that the decisions taken in the network at time $t$ are independent across agents. As a result,
$
    \Pr\left(\bigcap_{i\in S}\left\{ \lambda_i(t)=0\right\}\right) = \prod_{i\in S} (1-\gamma_i(t)).
$
In light of Definition~\ref{def:sma} and the Second Borel-Cantelli Lemma, this further implies that there exists an increasing sequence of random times $\{T_k\}_{k=1}^\infty$ such that $\lambda_i(T_{k}-1)=0$ \textit{a.s.\ }for all $i\in S$ and all $k\in\N$. This means that $ \blue x_i(T_k) = u_i$ \textit{a.s.\ }for all $i\in S$ and all $k\in\N$.

On the other hand, we can use the union bound to show that
${\Pr( \bigcup_{i\in S} \bigcup_{i\in \bar S} \{ a_{ij}(t)=1\} )}$ is at most \\$ \sum_{i\in S} \sum_{j\in \bar S} q_{ij}(t) =   \allone^T Q_{S\bar S}(t)\allone$. Since $\sum_{t=0}^\infty \allone^T Q_{S\bar S}(t)\allone < \infty$, it follows from the First Borel-Cantelli Lemma~\cite[Theorem 2.3.1]{durrett2019probability} 
that there exists a random time $T^*<\infty$ such that $a_{ij}(t)=0$ \textit{a.s.\ }for all $i\in S$, $j\in \bar S$, and $t\geq T^*$. This means that there \textit{a.s.} exists a point of time $T^*$ after which the agents in $S$ are never influenced by those in $\bar S$.

Let $K:=\inf\{k\in\N_0: T_k> T^*\}$. Then, $T_K>T^*$ and $\blue x_i(T_K)=u_i$ \textit{a.s.\ }for all $i\in S$. Hence, ${\{\blue x_i(t):i\in S\}\subseteq\U\bigcup\{\blue x_i(T_K):i\in S\} \subseteq \U}$ for all ${t\geq T_K}$.

It remains to show the existence of a time $T\geq T_K$ such that $\{\blue x_i(t):i\in\bar S\}\subseteq\U$ for all $t\geq T$. By the definition of ergodicity, the truncated chain $\{Q(t)\}_{t=\tau}^\infty$ is ergodic for all $\tau\in\N_0$. It follows from Remark~\ref{rem:matrix_limit} that there exists a random vector ${\blue\boldpi(\tau)\in\{\blue\ee^{(i)}:i\in[n]\}}$ such that $\lim_{t\to\infty}A(t:\tau) = \allone \blue\boldpi^T(\tau)$ \textit{a.s.} Thus, ${\lim_{t\to\infty}A(t:T_K) = \allone \blue\boldpi^T(T_K)}$ \textit{a.s.} On the other hand, for $t\geq T_K\geq T^*$, we have $a_{ij}(t)=0$ \textit{a.s.\ }
for all $i\in S$ and ${j\in \bar S}$. Hence, $A_{S\bar S}(t:T_K)=O$ \textit{a.s.\ }for all $t\geq T_K$. It follows that $\lim_{t\to\infty} A
_{S\bar S}(t:T_K)=O$ \textit{a.s.}, which means that $\blue\boldpi^T(T_K)\notin\{e^{(i)}:i\in\bar S\}$ \textit{a.s}. Since $\allone\blue\boldpi^T(T_K)$ has identical rows, this further implies that the columns of ${\lim_{t\to\infty} A(t:T_K)}$ indexed by $\bar S$ are all zero \textit{a.s.}  It now follows from Remark~\ref{rem:proof_aid} that 
\begin{align*}
    \limsup_{t\to\infty}B_{[n]\bar S}(t:T_K) &= \limsup_{t\to\infty}P_{[n]\bar S}(t:T_K)\cr 
    &\leq \limsup_{t\to\infty} A_{[n]\bar S}(t:T_K)=O\quad a.s.
\end{align*}
Equivalently, for all sufficiently large $t$, the entries of $\blue\y(t) = B(t:T_K)\blue\y(T_K)$ are binary convex combinations of ${\{\blue x_i(T_K):i\in  S\}\bigcup\U=\U}$. This completes the proof.
\end{proof}

We now consider a special case of~\eqref{eq:random_fj} in which the probability distributions of the agents' opinions converge to limits that can be computed using closed-form expressions.

\begin{theorem}\label{thm:fj_model}
Suppose the matrix pairs $\{(\Lambda(t), A(t))\}_{t=0}^\infty$ are independent and identically distributed. Also, suppose $\Gamma<I$, where $\Gamma:=\expect\Lambda(t)$ and $Q:=\expect A(t)$ for all $t\in\N_0$. Finally, suppose that $|\U|=n$ and that $\{\blue x_i(0)\}\bigcap\U=\emptyset$. Then the following assertions hold.
\begin{enumerate} [leftmargin=6mm,label={(\roman*)}]
    \item \label{item:one} We have ${\lim_{t\to\infty}\Pr(\blue x_i(t)=u_j) = v_{ij}}$ for all ${i,j\in [n]}$, where $\{v_{ij}:i,j\in[n]\}$ are the entries of
    \begin{align}
        V:=\left(I-\Gamma Q\right)^{-1}\left(I-\Gamma\right).
    \end{align}
    \item \label{item:two} There \textit{a.s.\ }exists a random time $T<\infty$ such that ${\blue x_i(t)\in\U}$ for all $t\geq T$.
\end{enumerate}
\end{theorem}
\begin{proof}
We first recall from Remark~\ref{rem:proof_aid} that ${\blue\y(t) = B(t:0) \blue\y(0)}$ for all $t\in\N_0$. Since $\U$ has $n$ distinct elements and since $\{\blue x_i(0)\}\bigcap \U=\emptyset$, this implies that
\begin{align}\label{eq:just_expect}
    \Pr(\blue x_i(t)=u_j) &= \Pr\left(B_{i\, n+j}(t:0)=1\right)
    = \E\left[B_{i\, n+j}(t:0)\right]\nonumber\\
    &\stackrel{(a)}= \expect{B}_{i\,n+j}(t:0)\stackrel{(b)}=\left(\left(\expect B(0)\right)^t\right)_{i\,n+j},
\end{align}
where $(a)$ and $(b)$ hold because $\{(\Lambda(t),A(t))\}_{t=0}^\infty$ are i.i.d. 

Thus, it suffices to evaluate $\lim_{t\to\infty} (\expect B(0))^t$. Observe that for the expected dynamics $\blue{\expect \x}(t+1) = \Gamma Q \expect {\blue\x}(t) + (I- \Gamma)\blue\u$, we have 
${\blue{\expect \y}(t) = (\expect B(0))^t  \blue\y(0)}$ as a consequence of Remark~\ref{rem:proof_aid}. On the other hand, we know from Theorem 21 and Corollary 22 in~\cite{proskurnikov2017tutorial} that $\lim_{t\to\infty} \blue{\expect \x}(t) = V \blue\u$, which implies that 
$$
    \lim_{t\to\infty} \blue{\expect \y}(t) =
    \begin{pmatrix}
        O_{n\times n} & V\\
        O_{n\times n} & I
    \end{pmatrix}
    \begin{pmatrix}
        \blue\x(0)\\
        \blue\u
    \end{pmatrix} =
    \begin{pmatrix}
        O_{n\times n} & V\\
        O_{n\times n} & I
    \end{pmatrix} \blue\y(0).
$$
That is, 
$ \lim_{t\to\infty} (\expect B(0))^t \blue\y(0)=
    \begin{pmatrix}
        O_{n\times n} & V\\
        O_{n\times n} & I
    \end{pmatrix} \blue\y(0).
$
As $\blue\y(0)$ (which stacks the initial states and the external influences) is arbitrary, it follows that ${\left(\lim_{t\to\infty}(\expect B(0))^t\right)_{i\,n+j} = v_{ij}}$ for all $i,j\in [n]$. In light of~\eqref{eq:just_expect}, this proves~\ref{item:one}.

To prove~\ref{item:two}, note that $\prod_{i\in [n]}(1- \gamma_i(t))$ is positive and time-invariant because $\Gamma(t)=\Gamma<I$ for all $t\in\N_0$. Hence, all the agents in the network are simultaneously malleable. Since they also dominate in expectation trivially,~\ref{item:two} follows immediately from Theorem~\ref{thm:malleable}.
\end{proof}

\subsection{Rank-One Perturbation of the Friedkin-Johnsen Variant}\label{sec:rank_one}

Another random adaptation-based variant of the Friedkin-Johnsen model can be obtained by letting the opinion of each agent `mutate' to any external influence with a fixed probability distribution, i.e., in the $t$th time period, agent $i$ either adapts to a neighbor's opinion or adapts to one of the prejudices independently of her past choices. With  probability $ \gamma_i(t)\in (0,1)$, the agent follows the adaptation scheme described earlier (in~\eqref{eq:main}), and with probability ${ 1 - \gamma_i(t)} $, however, instead of adapting to one fixed prejudice, she chooses an opinion from the set ${\U=\{u_i:i\in [n]\}}$, according to a stochastic vector $\blue{\mathbf q}$ on  $\U$. That is, 
\begin{align}\label{eq:rankone_dynamics}
    \blue\x(t+1) = \Lambda(t)A(t)\blue\x(t) + (I - \Lambda(t))C(t)\blue\u,
\end{align}
where $\blue\x(t)$ and $A(t)$ are as before, $\blue\u \in \R^n$ is the vector of external influences, $\{\Lambda(t)\}_{t=0}^\infty$ is a sequence of random binary diagonal matrices with $\Pr(\lambda_i(t) = 1) =  \gamma_i(t)$, and $\{C(t)\}_{t=0}^\infty$ is a sequence of i.i.d.\ binary stochastic random matrices. For any $i \in [n]$, $\Pr(C_{ij}(t) =1) =  q_j$ for all $j \in [n]$, independent of the other rows. Here $\expect C(t) = \allone \blue{\mathbf q}^T $, for all $t\in \N_0$, which is a rank-one matrix for all $t \in \N_0$.

\begin{theorem}\label{thm:rankone_perturb}
Consider the dynamics \eqref{eq:rankone_dynamics} where $|\U|=n$, $\{\blue x_i(0):i\in[n]\} \bigcap \U = \blue\emptyset$, and the matrix pairs $\{(\Lambda(t), A(t))\}_{t=0}^{\infty}$ are i.i.d.\ with $Q:= \expect A(t)$ and $\Gamma := \expect \Lambda(t)$ for all $t \in \N_0$. Also, suppose that $\Gamma< I$. Then the following hold true.
\begin{enumerate} [label={(\roman*)}, leftmargin =0.6cm]
    \item We have $\lim_{t\to \infty} \Pr(\blue x_i(t) = u_j) =  v_{ij}$ for all ${i,j \in [n]}$, where $\{v_{ij}:i,j \in[n]\}$ are the entries of 
    \begin{align*}
        V &:=  (I-\Gamma Q)^{-1}(I- \Gamma)\allone \blue {\mathbf q}^T.
    \end{align*}
    \item There \textit{a.s.\ }exists a random time $T< \infty$ such that ${\blue x_i(t) \in \U}$ for all $t\geq T$. \label{enum:thmPerturb_2nd}
\end{enumerate}
\end{theorem}

\begin{proof}
We can rewrite the dynamics as ${\blue{\y}(t+1) = D(t)\blue\y(t)}$, where $\blue{\y^T}(t) = \begin{bmatrix}\blue\x^T(t) & \blue\u^T\end{bmatrix}$ and 
\begin{align}
    D(t) &= 
    \begin{bmatrix}
    \Lambda(t) A(t) & (\blue{I}- \Lambda(t))C(t) \\
    O_{n\times n} & I_{n\times n}
    \end{bmatrix}.
\end{align}
So, we have $\blue\y(t) = D(t:0)\blue\y(0)$, for all $t\in \N_0$. Since the elements of $\U$ are distinct and $\{\blue x_i(0)\mid i\in[n]\}\bigcap \U = \blue\emptyset$, similar to ~\eqref{eq:just_expect}, we have $\Pr(\blue x_i(t) = u_j) = \left( \left( \expect D(0) \right)^t \right)_{i\,n+j}$. To compute the limiting marginal probability note that
\begin{align}
    \left( \expect D(0)\right)^t &= 
    \begin{bmatrix}
     \Gamma Q &    (\blue I-\Gamma) \expect C(0) \\
    O_{n\times n} & I_{n\times n}
    \end{bmatrix}^t 
    = 
    \begin{bmatrix}
      \left(\Gamma Q\right)^t &  R(t) \\
     O_{n\times n} & I_{n\times n}
     \end{bmatrix}, 
\end{align}
where ${ R(t) = \sum_{k=0}^{t-1} \left(\Gamma Q\right)^{k} (I-\Gamma) \allone \blue{\mathbf q}^T.}$ 
Since ${Q\in\P_{n\times n}}$, we know that the maximum absolute value of eigenvalues of $\Gamma Q$ is less than $1$ since $\Gamma Q \leq \gamma_{\max} Q$, where $\gamma_{\max}= \max_{i\in[n]} \gamma_i \in (0, 1)$. Using a result on Neumann Series (Eq.~(7.10.11) in \cite{meyer2000matrix}), we have ${\lim_{t \to\infty} R(t) = (I-\Gamma Q)^{-1}(I-\Gamma)\allone \blue{\mathbf q}^T}=V$, and $\lim_{t\to \infty} \left(\Gamma Q\right)^t = O_{n\times n}. $ 

For \ref{enum:thmPerturb_2nd}, note that the probability that all the agents adapt to an external influence at any time is ${\prod_{i=1}^n (1-\gamma_i)>0}$. Since $\{\Lambda(t)\}_{t=0}^\infty$ are independent, it follows from the Second Borel-Cantelli Lemma that there exists an increasing sequence of random times $\{T_k\}_{k=1}^{\infty}$ such that $\Lambda(T_k)=1$. This implies that $\blue x_i(T_1+1) \in \U$ for all $i\in [n]$. Therefore, for all $(t_0,\blue\x(t_0))\in\N_0\times \R^n$, there \textit{a.s.\ }exists a random time $T<\infty$ such that $\blue x_i(t) \in \U$ for all $t \geq T$. 
\end{proof}

\begin{remark}
Suppose, in addition to the assumptions in Theorem~\ref{thm:rankone_perturb}, all the agents have identical susceptibility, i.e., $ \Gamma = \gamma I$ for some $\gamma\in(0,1)$. Then, since ${Q \in\P_{n\times n}}$, we have ${R(t) = \sum_{k=0}^{\infty}  \gamma^k (1-\gamma) \allone \blue{\mathbf q}^T}$, and since ${\gamma \in (0,1)}$, we have ${V = \allone \blue{\mathbf p}^T}$. Furthermore, for this case, the result extends to all stochastic chains $\{Q(t)\}_{t=0}^\infty$ and not just to identically distributed chains, as ${R(t) = \sum_{k=0}^{t} \gamma^k (1-\gamma) Q(k+1:0) \allone \blue{\mathbf q}^T}$ with ${Q(t+1:0)\allone = \allone}$, for all $t\in \N_0$, which implies that ${V= \lim_{t\to\infty}R(t)= \allone \blue{\mathbf q}^T}$. Therefore, in this case, the limiting marginal probability distribution is independent of the degree of susceptibility.  
\end{remark}
\blue{\begin{remark}
Note that the dynamics studied in Sections~\ref{subsec:rv_fj} and \ref{sec:rank_one} can be obtained as special cases of the generalized model studied in \cite{ravazzi2014ergodic}. However, while the results of \cite{ravazzi2014ergodic} imply convergence in distribution and provide the expected values of the steady states, they do not characterize the distributions of these steady states, nor do they show the a.s. convergence of the agents' opinions to the prejudice set. 
\end{remark}
}

\section{SIMULATIONS}\label{sec:simulations}

We now illustrate some of our main results with the help of suitable numerical examples generated using MATLAB. These examples are aimed at facilitating the reader's understanding of the key ideas developed in this work. 

\begin{example} [\textbf{Base Case Dynamics with Ergodicity}]\label{eg:one} Consider Equation~\eqref{eq:main}, the simplest among all of the random adaptation dynamics we have analyzed above. Suppose we have $n=10$ agents in a social network, and for simplicity, suppose that the initial states (or opinions) of the agents are given by $x_i(0)=i$ for all $i\in[n]$.

In order to simulate the case of the expected stochastic chain $\{Q(t)\}_{t=0}^\infty$ being ergodic, we first assume a finite time horizon of $H=1000$, use \emph{randfixedsum.m}~\cite{rand} to randomly generate an indefinitely long (but finite) sequence of $n\times n$ stochastic matrices, select among the generated matrices the first $H$ matrices that are irreducible (i.e., the first $H$ that can be expressed as the adjacency matrices of strongly connected directed graphs), and then set the finite sequence $\{Q(t)\}_{t=0}^{H-1}$ equal to the sequence of the irreducible matrices thus obtained.

Next, we obtain $\{A(t)\}_{t=0}^{H-1}$ as the random matrices generated using the probability distributions defined by $\{Q(t)\}_{t=0}^H$. We then choose an arbitrary realization of $\{A(t)\}_{t=0}^{H-1}$ and plot the corresponding dynamics~\eqref{eq:main} in Figure~\ref{fig:base_case} below.
\end{example}

\begin{figure}[htp]
    \centering
    \includegraphics[scale=0.36]{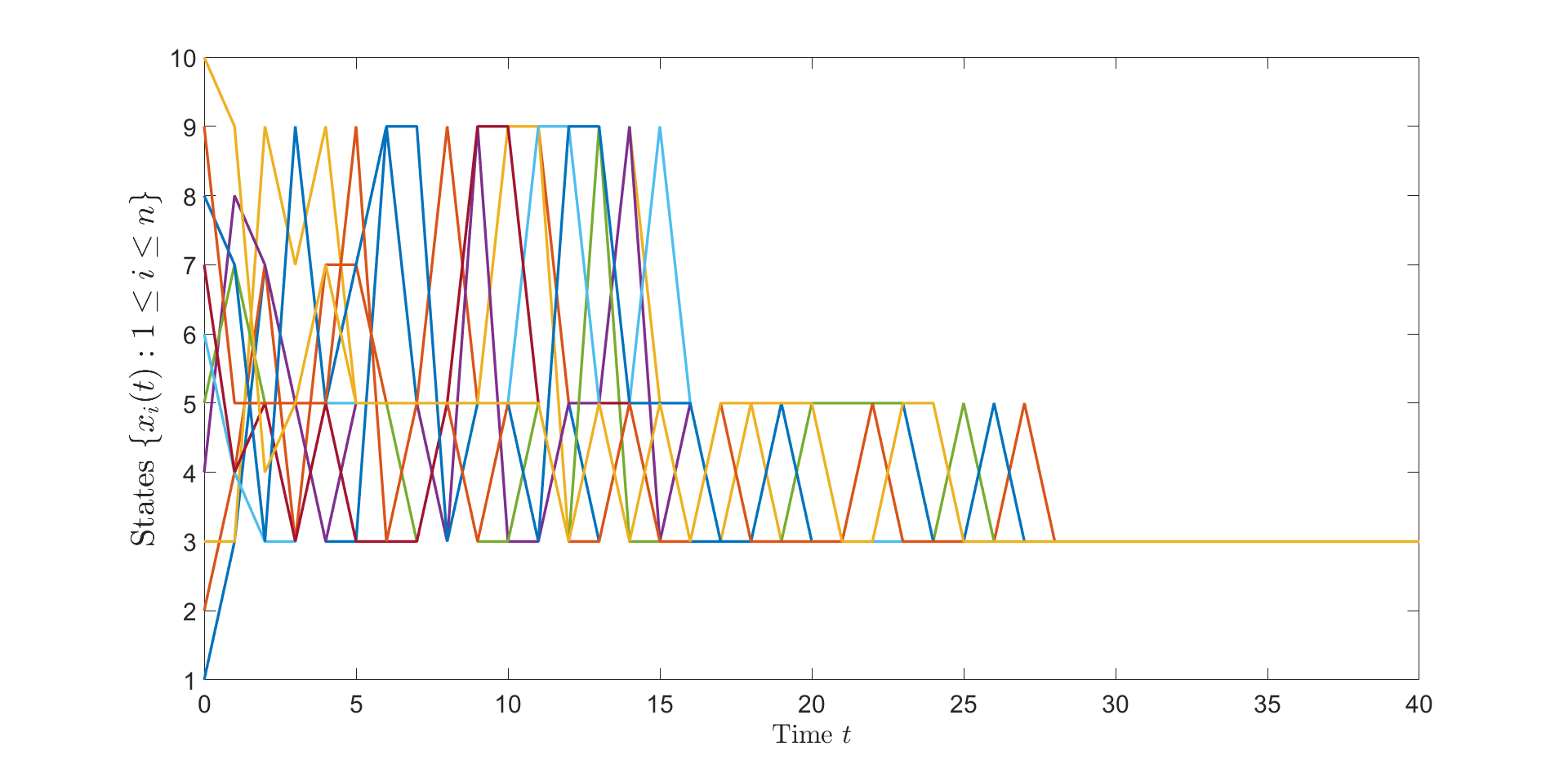}
    \caption{Simulation of the base case dynamics~\eqref{eq:main} with $\{Q(t)\}$ being an ergodic chain}
    \label{fig:base_case}
\end{figure}

In the above example, observe that all the agents' opinions reach the same value within the first $30$ time periods. This is in agreement with Theorem~\ref{thm:main}, which implies that finite agreement occurs almost surely when the expected stochastic chain $\{Q(t)\}=\{\expect A(t)\}$ is ergodic.

\begin{example} [\textbf{Base Case Dynamics without Ergodicity}] We again consider Equation~\eqref{eq:main} and repeat the procedure described in Example~\ref{eg:one}, except that we now generate $\{Q(t)\}_{t=0}^{H-1}$ as follows: we first generate two sequences $\{R(t)\}_{t=0}^{H-1},\{S(t)\}_{t=0}^{H-1}\subset\R^{\frac{n}{2}\times\frac{n}{2} }$ of irreducible matrices, we then generate two arbitrary sequences of random non-negative matrices $\{\Delta_1(t)\}_{t=0}^{H-1},\{\Delta_2(t)\}_{t=0}^{H-1}\subset\R^{\frac{n}{2}\times\frac{n}{2} }$, we let
$$
    \tilde Q(t) = 
    \begin{pmatrix}
    B(t) & \frac{1}{t^2}\Delta_1(t)\\
    \frac{1}{t^2}\Delta_2(t) & C(t)
    \end{pmatrix},
$$
and we finally set $q_{ij}(t) = \frac{\tilde q_{ij}(t) }{\sum_{k=1}^n \tilde q_{ik}(t)}$ for each $i,j\in [n]$ so as to obtain $Q(t)$ as a row-stochastic matrix. The decay factor of $\frac{1}{t^2}$ ensures that the influence of the agent subsets $\{1,2,\ldots, \frac{n}{2}\}$ and $\{\frac{n}{2} ,\frac{n}{2}+1,\ldots, n\}$ on each other diminishes with time fast enough to make $\{Q(t)\}_{t=0}^{H-1}$ a good approximation of a non-ergodic chain. 

Next, we obtain $\{A(t)\}_{t=0}^{H-1}$ as the random matrices generated using the probability distributions defined by $\{Q(t)\}_{t=0}^H$. We then choose an arbitrary realization of $\{A(t)\}_{t=0}^{H-1}$ and plot the corresponding dynamics~\eqref{eq:main} in Figure~\ref{fig:non-ergodic} below.
\end{example}

\begin{figure}[htp]
    \centering
    \includegraphics[scale=0.36]{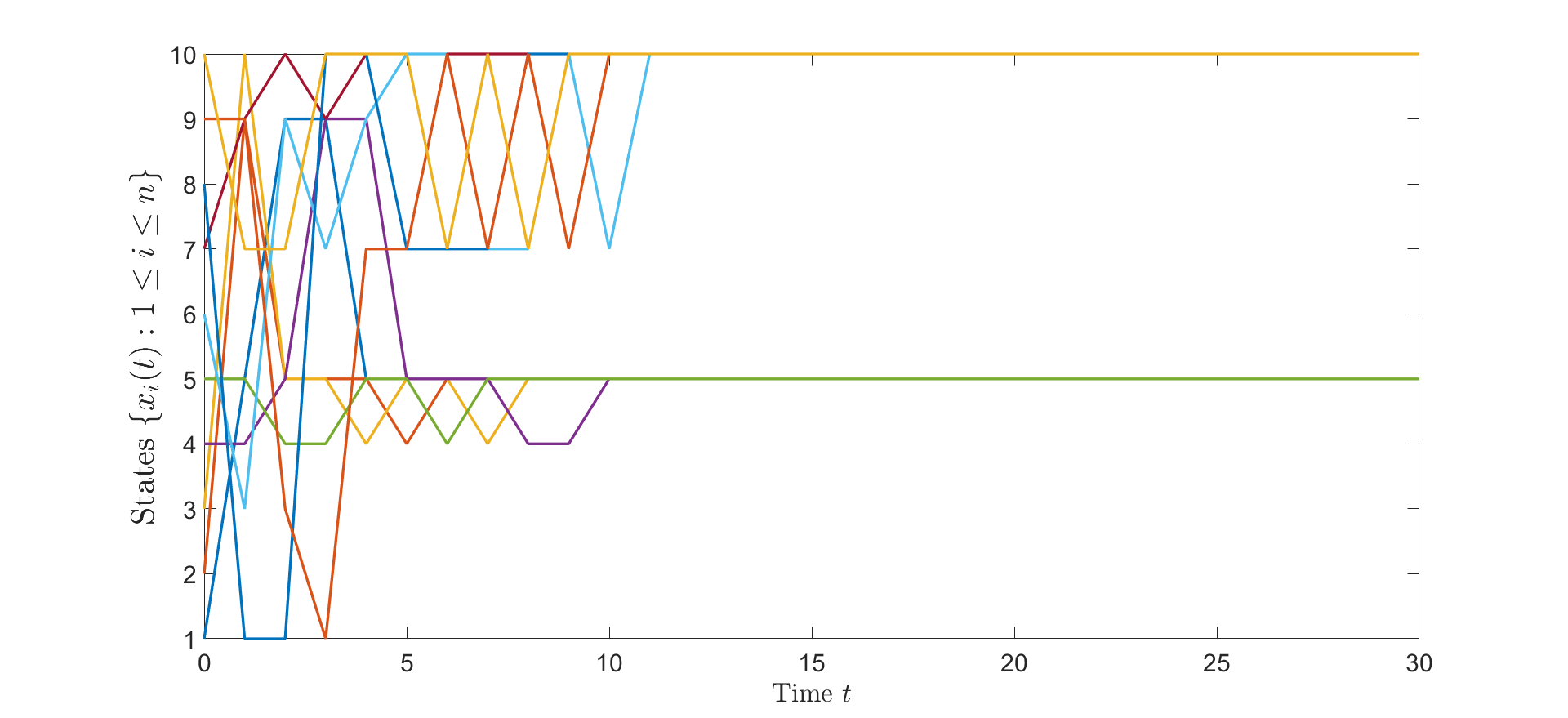}
    \caption{Simulation of the base case dynamics~\eqref{eq:main} with $\{Q(t)\}$ being a non-ergodic chain}
    \label{fig:non-ergodic}
\end{figure}

In the above example, two clusters of agents reach two distinct consensus values within the first $15$ time periods. This is consistent with Theorem~\ref{thm:main}, which states that finite agreement almost surely does not occur in the base case dynamics if $\{Q(t)\}$ is not an ergodic chain. Moreover, it can be shown that the principal submatrices of $\tilde Q(t)$ corresponding to the coordinates $\{1,2,\ldots,\frac{n}{2}\}$ form an ergodic stochastic chain. The same can be said about the principal submatrices corresponding to the complementary set $\{\frac{n}{2},\ldots,n-1, n\}$, which, along with Theorem~\ref{thm:main}, explains why exactly two consensus clusters are formed in the associated random adaptation dynamics.

\begin{example} [\textbf{Comparison of the Base Case Random Adaptation Dynamics with the Expected Dynamics}] We now repeat the setup of Example~\ref{eg:one} and plot in Figure~\ref{fig:comparison} the evolution of $\x(t)$ averaged over $1000$ different realizations of $\{A(t)\}_{t=0}^{H-1}$ as a set of $n$ line plots. To compare the resulting plots with the associated expected dynamics, i.e., the classical non-random averaging dynamics defined by~\eqref{eqn:averaging}, we simulate~\eqref{eqn:averaging} and plot the corresponding state evolution of the $n$ agents using solid red square markers whose color intensity varies with the agent index $i\in[n]$.
\end{example}
\begin{figure}[htp]
    \centering
    \includegraphics[scale=0.36]{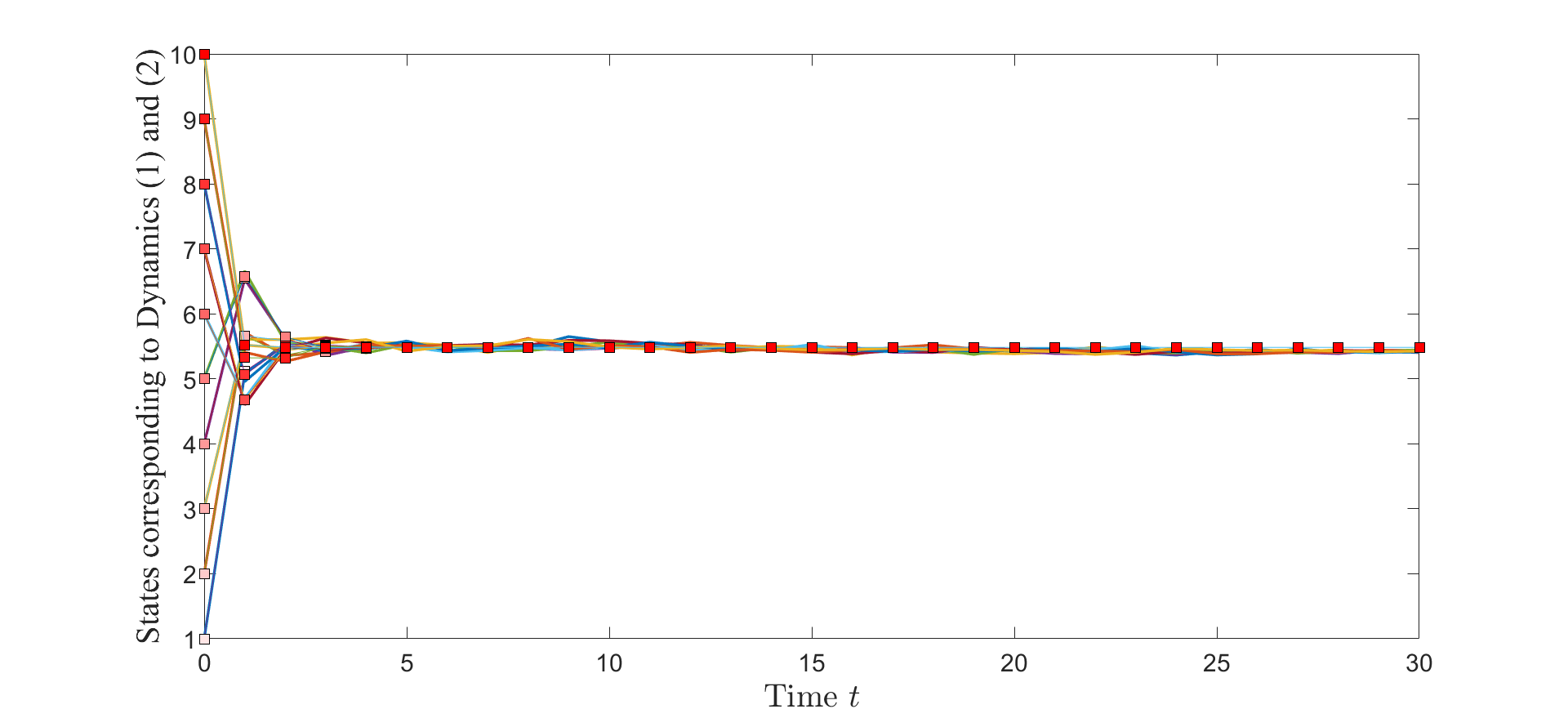}
    \caption{Comparison of the Empirical Mean of the Base Case Dynamics with the Expected Dynamics~\eqref{eqn:averaging}}
    \label{fig:comparison}
\end{figure}

In Figure~\ref{fig:comparison}, we observe that the line plots and the square marker plots exhibit finite agreement and also overlap with each other almost perfectly. This is consistent with Theorem~\ref{thm:main} and with the observation that 
$$
    \E[\x(t)]=\E[A(t:t_0)\x(t_0)]=\E[A(t:t_0)]\x(t_0) =  Q(t:t_0)\x(t_0),
$$
where the last step follows from the mutual independence of $\{A(t)\}_{t=0}^\infty$. This observation means that the distributed averaging dynamics defined by~\eqref{eqn:averaging} is nothing other than the expectation of the random adaptation dynamics defined by~\eqref{eq:main}. Hence, the empirical average of $\x(t)$ plotted using solid lines must be a good approximation for  the averaging dynamics~\eqref{eqn:averaging}.

\begin{example} [\textbf{Random Adaptation Variant of Friedkin-Johnsen Dynamics}] Finally, we simulate~\eqref{eq:random_fj}, the random adaptation variant of the Friedkin-Johnsen model. We first set $n=10$ and  generate $\{Q(t)\}_{t=0}^{H-1}=\{\expect A(t)\}_{t=0}^{H-1}$ using the procedure described in Example~\ref{eg:one}. We then generate $\{\gamma_i(0):i\in [n]\}$ as random variables uniformly distributed over the set $[0,1]$, and we make sure that $0<\gamma_i(0)<1$ for all $i\in[n]$. Next, we set $\gamma_i(t)=\gamma_i(0)$ for all $i\in[n]$ and all $t\in\{0,1,\ldots, H-1\}$ so as to ensure that $\{\Lambda(t)\}_{t=0}^{H-1}$ are i.i.d. random matrices. As the next step, we generate $\lambda_i(t)=\text{Bernoulli}(\gamma_i(0))$ for each $i\in[n]$ and $t\in\{0,1,\ldots, H-1\}$.

As for the initial opinions and prejudices, we choose $u_i=i+20$ and $x_i(0)=i$ for all $i\in [n]$ to ensure that $\{x_i(0):i\in [n]\}\bigcap \mathcal U=\emptyset$. We then plot the results of our simulation in Figure~\ref{fig:fj} below.
\end{example}

\begin{figure}[htp]
    \centering
    \includegraphics[scale=0.36]{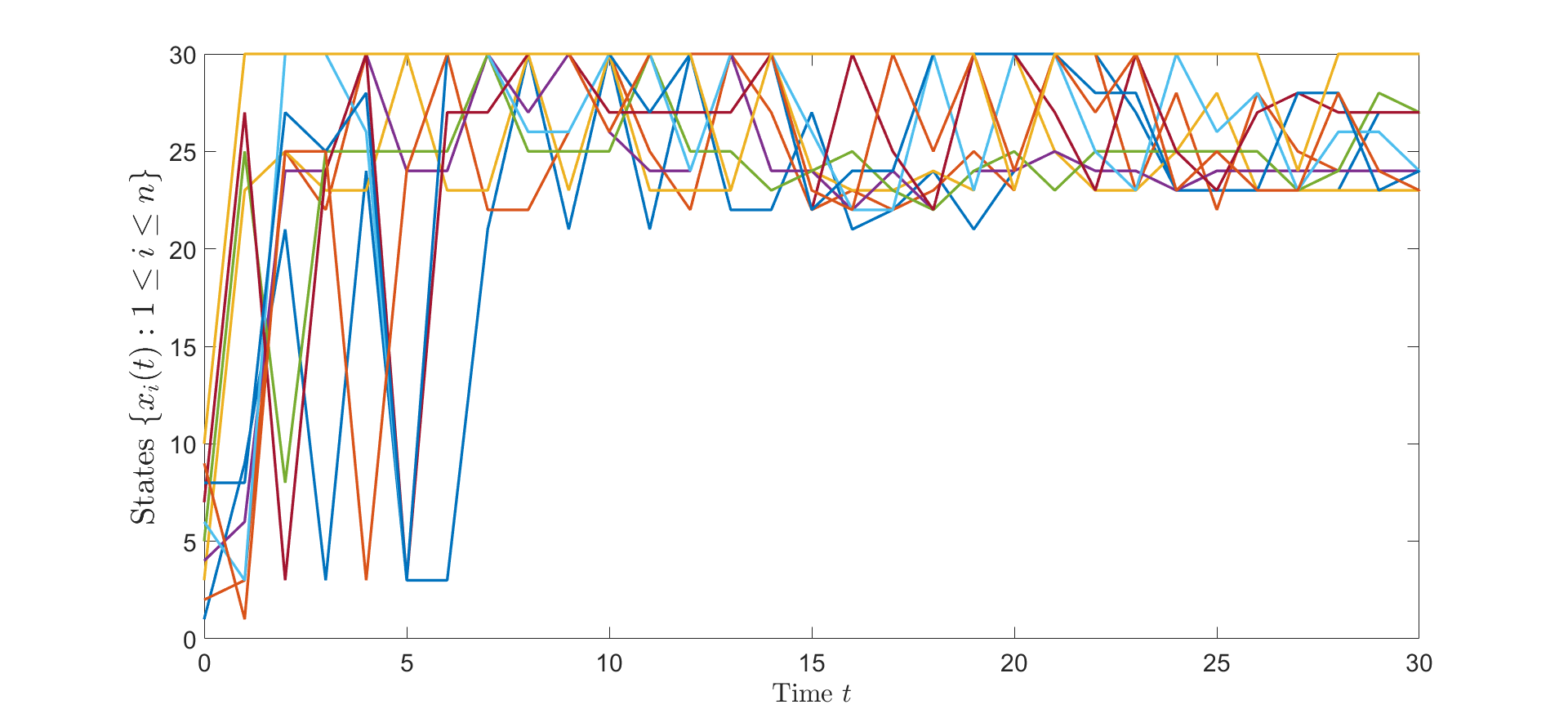}
    \caption{Simulation of the Random Adaptation Variant~\eqref{eq:random_fj} of the Friedkin-Johnsen Model}
    \label{fig:fj}
\end{figure}

As seen from Figure~\ref{fig:fj}, although the initial opinions (states) of all the agents lie in the set $\{1,2,\ldots, 10\}$, all the agents eventually enter the prejudice set $\{21,22,\ldots,30\}$. This happens because $\{\Lambda(t)\}_{t=0}^{H-1}$ are i.i.d., which implies that $[n]$ is a set of simultaneously malleable agents that dominate  in expectation trivially. Therefore, by Theorem~\ref{thm:malleable}, all the initial opinions of the agents get erased within a finite time-span.

In summary, all of our simulation results are consistent with, and hence validate, our main theoretical results.

\section{CONCLUSION}\label{sec:conclusion}
We proposed and studied a random adaptation variant of time-varying distributed averaging dynamics in discrete time. We have shown that our models give rise to novel interpretations for the concepts of ergodicity and absolute probability sequences, both of which are pivotal to the study of stochastic chains. We have also proposed a time-varying, stochastic analog of the well-known Friedkin-Johnsen opinion dynamics and analyzed the asymptotic behavior of the probability distributions of the agents' opinions. Finally, we have considered a rank-one perturbation of our base-case stochastic dynamics and studied its asymptotic behavior.

\blue{Our random adaptation interpretation of averaging (and related) dynamics opens up many avenues avenues for further investigations including the  (time-varying) controlled variation of Friedkin-Johnsen dynamics through the lens of random adaptation dynamics, and the connection of mutation-adaptation learning dynamics (in game theoretic setting) and our proposed random adaptation schemes.}

\section*{Appendix}

\begin{lemma} \label{lem:positive_correlation}
 For the dynamics~\eqref{eq:main} \blue{  with $Q(t)=\expect A(t)$}, the following holds for all $S\subseteq [n]$, $\ell\in[n]$, $t\in\N_0$ and $\Delta \in\N$
 \begin{align}\label{eq:pos_cor}
     \Pr\left(\bigcap_{i\in S} \left(A(t+\Delta:t)\right)_{i\ell}=1\right)\geq \prod_{i\in S}Q_{i\ell}(t+\Delta: t).
 \end{align}
 \end{lemma}
 \begin{proof}
 We use induction on $\Delta$.  For $\Delta=1$, the assertion follows from the independence of the rows of $A(t)$.

 Assume now that the assertion holds for some $\Delta\geq 1$. Let $S\subseteq[n]$ be a given set. W.l.o.g.\ we assume $S=[v]$ for some $v\in[n]$. In addition, for each $i\in[n]$, let $\sigma_i(t)$ denote the random index such that $a_{i\sigma_i(t)}(t)=1$, let $\tilde \sigma(t):=(\sigma_1(t),\ldots, \sigma_v(t))$ and let $\tilde \alpha=(\alpha_1,\ldots, \alpha_v)$ be a realization of $\tilde{\sigma}(t)$. Let ${W:=\{i\in S: \alpha_i\notin\{\alpha_1,\ldots, \alpha_{i-1}\}\}}$ index all the distinct $\alpha_i$s. Then, for $\Delta+1$, we have 
 $\bigcap_{i\in S}\left\{A_{i\ell}(t+\Delta+1:t)=1\right\}=\bigcap_{i\in S}\{A_{i\ell}(t+\Delta+1:t)=1\}\cap\{\sigma(t) = \tilde \alpha\}$.
 Since, $A(t)$ has independent rows and because $\{(A(t+\Delta+1:t))_{i\ell}=1\}=\{ (A(t+\Delta:t))_{\alpha_i\ell}=1\}$ conditional on the event $\{\sigma_i(t)=\alpha_i\}$, we have 
 \begin{align*}
     \Pr\left(\bigcap_{i\in S}\left\{A_{i\ell}(t+\Delta+1:t)=1\right\}\right)&{=}\sum_{\tilde \alpha\in[n]^v}\bigg(\Pr(\bigcap_{i\in S}\{ A_{\alpha_i \ell}(t+\Delta:t)=1 \}\mid \tilde \sigma(t)=\tilde \alpha)\cr
     &\quad\quad\quad\quad\quad\times\prod_{j\in S} q_{j\sigma_j(t)}(t)\bigg)\cr
     &\stackrel{(a)}{=}\sum_{\tilde \alpha\in[n]^v}\Pr(\bigcap_{i\in S}\{ A_{\alpha_i \ell}(t+\Delta:t)=1 \})\cdot\prod_{j\in S} q_{j\sigma_j(t)}(t)\cr
     &=\sum_{\tilde \alpha\in[n]^v}\Pr(\bigcap_{i\in W }\{ A_{\alpha_i \ell}(t+\Delta:t)=1 \})\cdot\prod_{j\in S} q_{j\sigma_j(t)}(t)\cr    &\stackrel{(b)}{\geq} \sum_{\tilde \alpha\in [n]^v} \prod_{i\in W}  Q_{\alpha_i \ell}(t+\Delta:t) \cdot\prod_{j\in S} q_{j\sigma_j(t)}(t)\cr
     &\stackrel{(c)}{\geq} \sum_{\tilde \alpha\in [n]^v} \prod_{i\in S}  Q_{\alpha_i \ell}(t+\Delta:t) \cdot\prod_{j\in S}  q_{j\sigma_j(t)}(t)
 \end{align*}
 where $(a)$ holds because $\{A(\tau)\}_{\tau=0}^\infty$ are independent, $(b)$ follows from our inductive hypothesis, and $(c)$ holds because $W\subseteq V$ and because all matrix entries lie in $[0,1]$. Noting that the last expression is simply  $\prod_{i\in S}  Q_{i\ell}(t+\Delta+1:t)$ completes the proof. 
\end{proof}

\bibliographystyle{ieeetr}
\bibliography{bib}

\end{document}